\documentclass[12pt]{amsart}
\usepackage{amsmath}
\usepackage{amsthm}
\usepackage{amsfonts}
\usepackage{amssymb}
\newtheorem{Theorem}{Theorem}[section]

\theoremstyle{definition}

\theoremstyle{remark}

\numberwithin{equation}{section}

\newcommand{\Z}{{\mathbb Z}}
\newcommand{\R}{{\mathbb R}}

\newcommand{\N}{{\mathbb N}}

\begin{document}

\title[Uniqueness of reflectionless Jacobi matrices]{Uniqueness of reflectionless
Jacobi matrices and the Denisov-Rakhmanov Theorem}

\author{Christian Remling}

\address{Mathematics Department\\
University of Oklahoma\\
Norman, OK 73019}

\email{cremling@math.ou.edu}

\urladdr{www.math.ou.edu/$\sim$cremling}

\date{June 14, 2010}



\thanks{CR's work has been supported
by NSF grant DMS 0758594}
\begin{abstract}
If a Jacobi matrix $J$ is reflectionless on $(-2,2)$ and has a single $a_{n_0}$ equal to
$1$, then $J$ is the free Jacobi matrix $a_n\equiv 1$, $b_n\equiv 0$. I'll discuss this
result and its generalization to arbitrary sets and present several applications, including the following:
if a Jacobi matrix has some portion of its $a_n$'s close to $1$,
then one assumption in the Denisov-Rakhmanov Theorem can be dropped.
\end{abstract}
\maketitle
\section{Statement of results}
A \textit{Jacobi matrix }is
a difference operator of the following type:
\[
(Ju)_n = a_nu_{n+1}+a_{n-1}u_{n-1}+b_nu_n
\]
Here, $a_n>0$ and $b_n\in\R$, and we also always assume that $a,b$ are bounded sequences.

Alternatively, one can represent $J$ by a tridiagonal matrix with respect to
the standard basis of $\ell^2(\Z)$:
\[
J = \begin{pmatrix} \ddots & \ddots & \ddots &&&& \\ & a_{-2} & b_{-1} & a_{-1} &&&\\
&& a_{-1} & b_0 & a_0 && \\ &&& a_0 & b_1 & a_1 & \\ &&&& \ddots & \ddots & \ddots
\end{pmatrix}
\]
Half line operators $J_+$, on $\ell^2(\Z_+)$, say, are defined similarly,
by considering a suitable truncation of this matrix.

The case $a_n\equiv 1$ is of particular interest; these operators are called
(discrete) \textit{Schr\"odinger operators.}

The Denisov-Rakhmanov (DR) Theorem \cite{Den,Rakh} says the following:
\begin{Theorem}[\cite{Den}]
Suppose that:\\
(1) $\sigma_{ess}(J_+)=[-2,2]$; (2) $\Sigma_{ac}(J_+)=[-2,2]$.\\
Then $a_n\to 1$, $b_n\to 0$ as $n\to\infty$.
\end{Theorem}
Here, $\Sigma_{ac}(J_+)$ denotes an essential support
of the absolutely continuous part of the spectral measure $\nu_+$ of $J_+$;
in other words, $\nu_{+,ac}(M)=0$ if and only if $|M\cap\Sigma_{ac}|=0$. Note that
$\Sigma_{ac}$ is determined up to sets of Lebesgue measure zero by this condition.
The absolutely continuous spectrum, $\sigma_{ac}$, can be obtained from $\Sigma_{ac}$
by taking the essential closure (of an arbitrary representative).

In general, both assumptions of the DR Theorem are needed.
However, in the Schr\"o\-din\-ger case ($a_n\equiv 1$), assumption (2) \textit{can }be dropped.
This surprising result is due to Damanik, Hundertmark, Killip, and Simon \cite{DHKS}; it predates
Denisov's work.

In this note, I'll show that it is also
possible to drop assumption (1) instead, and this in fact works in more general situations,
not just for Schr\"odinger operators. The basic tool will be a simple,
but perhaps also somewhat surprising observation about reflectionless Jacobi matrices.
Recall that a whole line Jacobi matrix $J$ is called \textit{reflectionless }on
a Borel set $M\subset\R$ of positive Lebesgue measure if
\[
\textrm{\rm Re }g_n(t) = 0 \quad \textrm{ for Lebesgue almost every }t\in M
\]
and for all $n\in\Z$. Here, $g_n(z)=\langle{\delta_n, (J-z)^{-1}\delta_n}\rangle$ is
the Green function of $J$ at site $n$, and $g_n(t)=\lim_{y\to 0+} g_n(t+iy)$;
this limit exists for almost every $t\in\R$. See \cite{Remac} for more
information on reflectionless operators and why they are important.

We will denote the set of reflectionless (on $M$) Jacobi matrices by $\mathcal R(M)$.
\begin{Theorem}
\label{T1.1}
Suppose that $J\in\mathcal R(I)$ with $I=[B-2A,B+2A]$. Then $a_n\ge A$ for all $n\in\Z$. If
$a_{n_0}=A$ for a single $n_0\in\Z$, then $a_n= A$, $b_n= B$ for all $n\in\Z$.
\end{Theorem}
Note that we don't make any assumptions on the spectral behavior of $J$
outside the interval $I$; the spectrum of $J$ could be much bigger than this set,
and there is no restriction on the type of this additional spectrum (if present).
$\mathcal R(I)$ is a large space, and it contains many familiar operators, for example
solitons and all periodic operators whose spectrum contains $I$.

Theorem \ref{T1.1} implies the
promised Schr\"odinger version of the DR Theorem without hypothesis (1).
The proof is very easy if we follow the treatment of the DR Theorem that was given in \cite{Remac}.
The idea was to show that the $\omega$ limit set, defined as the collection of
limit points $\lim S^{n_j}J_+$, consists of the free Jacobi matrix ($a_n\equiv 1$,
$b_n\equiv 0$) only. Here, $S$ denotes the shift map, and we take limits with
respect to pointwise convergence of the coefficients.
Please see again \cite{Remac} for more background information.

So suppose now that $a_n\equiv 1$ and $\Sigma_{ac}(J_+)=[-2,2]$,
and let $J_0\in\omega(J_+)$. Then clearly
$J_0$ has coefficients $a_n(J_0)\equiv 1$ also; moreover, $J_0\in\mathcal R([-2,2])$ by
\cite[Theorem 1.4]{Remac}. Hence $b_n(J_0)\equiv 0$ by Theorem \ref{T1.1}, which is what
we wanted to show.

The same argument establishes the following stronger result.
\begin{Theorem}
\label{TDR}
Suppose that $\Sigma_{ac}(J_+)\supset [-2,2]$. Furthermore,
suppose that there exists a subsequence $n_j\to\infty$ with bounded gaps
(that is, $\sup (n_{j+1}-n_j)<\infty$) so that $a_{n_j}\to 1$. Then
\[
a_n\to 1, b_n\to 0 \quad (n\to\infty) .
\]
\end{Theorem}
Here's another (related) immediate consequence of Theorem \ref{T1.1}:
\begin{Theorem}
\label{T1.2}
Suppose that $\Sigma_{ac}(J_+)\supset [B-2A,B+2A]$. Then
\begin{equation}
\label{1.1}
\liminf_{n\to\infty} a_n \ge A .
\end{equation}
\end{Theorem}
\begin{proof}
By \cite[Theorem 1.4]{Remac}, every $J_0\in\omega(J_+)$ satisfies $J_0\in\mathcal R(I)$.
By Theorem \ref{T1.1}, this implies that $a_n(J_0)\ge A$ for all $n\in\Z$. If \eqref{1.1} didn't
hold, then there would be $\omega$ limit points whose coefficients violate this inequality.
\end{proof}
It was previously known that $\liminf (a_1\cdots a_n)^{1/n}\ge A$ in the situation of
Theorem \ref{T1.2}; see \cite[Theorem 5.7]{DeiftSim} for a closely related statement,
and \cite{Simpot} for the use of potential theoretic tools in this context.

Theorem \ref{T1.1} also generalizes certain uniqueness results for
ergodic operators from \cite{DeiftSim}, but we won't make this explicit.
Finally, by making use of basic continuity properties, we can formulate
an Oracle Theorem type version of Theorem \ref{T1.1}.
\begin{Theorem}
\label{T1.1oracle}
Fix $I$ as above and $R>|B|+2A$. Then, for every $\epsilon>0$ and $L\in\N$, there exists $\delta>0$
such that if $J\in\mathcal R(I)$, $\|J\|\le R$,
and $a_{n_0}<A+\delta$ for some $n_0\in\Z$, then
\[
A\le a_n < A+\epsilon, \quad |b_n-B|<\epsilon \quad\quad (n_0-L \le n \le n_0+L) .
\]
\end{Theorem}
Notice that this is uniform in the sense that
$\delta$ depends on $\epsilon,L,R,A$, but not on $J$ (or $n_0$, $B$).
The same statement is obtained for all sufficiently large $n_0$ if instead of
$J\in\mathcal R(I)$, we just assume that $\Sigma_{ac}(J_+)\supset I$; this variant
immediately follows from \cite[Theorem 1.4]{Remac} again, which shows that
$\omega(J_+)\subset\mathcal R(I)$. Here's still another way of saying this: If $\Sigma_{ac}(J_+)
\supset I$ and $a_{n_j}\to A$ along some subsequence $n_j\to\infty$, then there are $L_j\to\infty$
so that
\[
\lim_{j\to\infty} \sup \left\{ \left| a_k - A \right|+\left| b_k-B\right| : |k-n_j|\le L_j \right\} = 0
\]
(if $a_{n_0}$ is close to $A$, then $(a_n,b_n)$ are close to $(A,B)$ on a long
interval centered at $n_0$).
Clearly, Theorem \ref{TDR} is an immediate consequence of this in turn;
all these results are closely related.

I'll present the proofs of this and Theorem \ref{T1.1} in the following section.
In the final section, I'll discuss a generalization of Theorem \ref{T1.1},
where we replace the interval $I$ by a general compact set.
This yields a general version of Theorem \ref{T1.2},
and this may be interpreted as a quantitative version
of the Theorem of Dombrowski and Simon-Spencer \cite{Dom,SimSp} on the absence of
absolutely continuous spectrum. We'll discuss all this in Section 3.
\section{Proof of Theorems \ref{T1.1}, \ref{T1.1oracle}}
We will make use of the inverse spectral theory for reflectionless
operators in the form presented in \cite{PR2,Remac} (but see also
\cite{Craig,Kot,Teschl}, among others, where such tools were used much earlier).
I will review some aspects of this theory, but please also consult
\cite{PR2,Remac,Teschl} for further details.

\begin{proof}[Proof of Theorem \ref{T1.1}]
By rescaling and shifting, it of course suffices to discuss
the case $A=1$, $B=0$. So let $J\in\mathcal R([-2,2])$.
Consider its $H$ function $H(z)=-1/g_0(z)$ (compare \cite{PR2,Remac}) and the associated
Krein function
\[
\xi(t) = \frac{1}{\pi}\lim_{y\to 0+} \textrm{Im }\ln H(t+iy) .
\]
Since $J$ is reflectionless on $(-2,2)$, we have that $\xi=1/2$ on this set;
see, for example, \cite[Proposition 2.1]{PR2}. Fix $R>0$ so large that $\sigma(J)\subset [-R,R]$.
Then the exponential Herglotz representation of $H$ reads
\begin{equation}
\label{2.3}
H(z) = (z+R)\exp\left( \int_{-R}^R \frac{\xi(t)\, dt}{t-z} \right) .
\end{equation}
The $H$ function $H_0(z)=\sqrt{z^2-4}$ of the free Jacobi matrix ($a_n\equiv 1$, $b_n\equiv 0$)
has a similar representation; here we can take $R=2$, and,
of course, $\xi_0(t)=\xi(t)=1/2$ on $t\in (-2,2)$. Thus
\[
H(z) = H_0(z) h(z) ,
\]
where
\begin{equation}
\label{2.1}
h(z) = \exp \left( \int_{-R}^{-2} \frac{\xi(t)-1}{t-z}\, dt + \int_2^R \frac{\xi(t)\, dt}{t-z}\right) .
\end{equation}
As a Herglotz function, $H$ has a unique associated (finite, compactly supported) measure $\rho$;
for example, we can obtain $\rho$ as the weak-$*$ limit
\begin{equation}
\label{2.4}
d\rho(t) = \frac{1}{\pi}\, \lim_{y\to 0+} \textrm{\rm Im }H(t+iy)\, dt .
\end{equation}
Since $J\in\mathcal R([-2,2])$, we have that $\textrm{Im }H(t)=|H(t)|$ almost
everywhere on $(-2,2)$, and this function is the density of $\pi\rho_{ac}$ on this set.
The Jacobi matrices from $\mathcal R([-2,2])$ are in one-to-one correspondence
with the half line spectral measures $\nu_+$ of the following type:
\[
d\nu_+(t) = \frac{1}{2}\chi_{(-2,2)}(t)\, d\rho_{ac}(t) + f(t)\, d\rho(t) ,
\]
where $0\le f\le 1$ and $f=0$ on $(-2,2)$. See \cite[Section 2]{PR2} and/or \cite[Section 5]{Remac}.
Rewrite this as
\[
d\nu_+(t) = \frac{1}{2\pi}\chi_{(-2,2)}(t) |H(t)|\, dt + d\mu(t) .
\]
If $J=J_0$ is the free Jacobi matrix, then the corresponding measure is given by
\[
d\nu_+^{(0)}(t) = \frac{1}{2\pi}\chi_{(-2,2)}(t) |H_0(t)|\, dt ,
\]
so we can also say that
\begin{equation}
\label{2.2}
d\nu_+(t) = |h(t)|\, d\nu_+^{(0)}(t) + d\mu(t) .
\end{equation}
Now by direct inspection of \eqref{2.1}, we easily confirm that $|h(x)|\ge 1$ for $-2<x<2$;
indeed, both integrals from the exponent are non-negative for these $z=x$. Thus
\[
\nu_+(\R) \ge \nu_+^{(0)}(\R) = 1 .
\]
We have the general formula $\nu_+(\R)=a_0^2$; compare the discussion of formula (2.3)
from \cite{PR2} or see Sections 2.1, 2.7 of \cite{Teschl}. This gives the inequality
stated in Theorem \ref{T1.1} (for $n=0$, but of course we can consider shifts of $J$
to obtain the statement for
general $n\in\Z$). Moreover, if $a_0=1$, then $|h(x)|=1$ almost everywhere on
$x\in(-2,2)$, but then another look at \eqref{2.1} reveals that
$\xi=1$ on $(-R,-2)$, $\xi=0$ on $(2,R)$.
So $h\equiv 1$ and thus $H=H_0$, $\nu_+=\nu_+^{(0)}$, and since $\nu_+$ determines $J$, as mentioned
earlier, we conclude that $J$ is the free Jacobi matrix $J_0$.
\end{proof}

\begin{proof}[Proof of Theorem \ref{T1.1oracle}]
To obtain Theorem \ref{T1.1oracle} from this argument, notice that
\[
\left\| \nu_+ -\nu_+^{(0)}\right\| = \nu_+(\R)-\nu_+^{(0)}(\R) = a_0^2-1 .
\]
This follows because \eqref{2.2}
shows that $d\nu_+^{(0)}=g\, d\nu_+$ with $g\le 1$.
The map $\nu_+\mapsto J$, restricted to those $\nu_+$ for which the corresponding $J\in\mathcal R(I)$
satisfies $\|J\|\le R$, is uniformly continuous if we use a metric that induces the weak-$*$ topology on
the $\nu_+$ and pointwise convergence of the coefficients for the Jacobi matrices $J$;
explicitly, we may use the metric
\begin{equation}
\label{defd}
d(J,J') = \sum_{n\in\Z} 2^{-|n|}\left( |a_n-a'_n|+|b_n-b'_n| \right)
\end{equation}
on the image.
See \cite[Proposition 2.3]{PR2} for these statements. Theorem \ref{T1.1oracle} follows.
\end{proof}
\section{Reflectionless operators on general sets}
I now present a generalization of Theorem \ref{T1.1}; instead of intervals
$I$, we now consider arbitrary compact sets $K$. We will of course insist that $K$ has
positive Lebesgue measure; otherwise, the condition of being reflectionless
on $K$ becomes vacuous.
We must first decide on a proper replacement for the Jacobi matrix
with constant coefficients $a_n=A$, $b_n=B$.
It is natural to proceed as follows. For a compact set $K\subset\R$, define
\[
\mathcal R_0(K)=\{ J\in\mathcal R(K): \sigma(J)\subset K \} .
\]
This is a compact set itself if we again use the metric $d$ from \eqref{defd}.
If $K=I=[B-2A,B+2A]$, then $\mathcal R_0(K)$ consists of a single operator, the Jacobi matrix
with constant coefficients.

Please see also \cite{PR1} and especially
\cite{PR2} for a rather extensive discussion of various aspects of these spaces.
\begin{Theorem}
\label{T3.1}
Fix a compact set $K\subset\R$ of positive Lebesgue measure. Then there exists a constant
$A=A(K)>0$ such that the following holds:
If $J\in\mathcal R(K)$, then $a_n\ge A$ for all $n\in\Z$. Moreover, if $a_{n_0}=A$ for
a single $n_0\in\Z$, then $J\in\mathcal R_0(K)$ (and there are such
minimizing operators $J_0\in\mathcal R_0(K)$ for which $a_{n_0}(J_0)=A$).
\end{Theorem}
As we will see, this is easy to prove with the same methods. Note, however, that
the minimizing $J_0$'s can easily form a small subset of
$\mathcal R_0(K)$. There are simple examples (periodic operators)
where most $J\in\mathcal R_0(K)$ satisfy the strict inequalities $a_n(J)>A$.

It would be interesting to investigate if one could also treat all operators from
$\mathcal R_0(K)$ on an equal footing,
at least in certain well-studied cases (finite gap sets $K$). One could try to
find a function $F(a,b)$ of the coefficients that characterizes $\mathcal R_0(K)$ among all
operators from $\mathcal R(K)$ in the sense that
$F(S^na,S^nb)\ge F_0$ for all $J\in\mathcal R(K)$ and all $n\in\Z$ ($S$ denotes the shift), and
we have equality for a single $n_0\in\Z$ precisely if $J\in\mathcal R_0(K)$. This can reasonably be
expected to hold only if $K$ is an essentially closed set (consider solitons!).

Theorem \ref{T3.1} does include Theorem \ref{T1.1} as a special case, because, as discussed,
if $K=[B-2A,B+2A]$, then only the Jacobi matrix with constant coefficients $a_n=A$, $b_n=B$
lies in $\mathcal R_0(K)$.
\begin{proof}
We follow the strategy from Section 2, suitably adjusted. We can simply define
\[
A = \inf_{J\in\mathcal R_0(K)} a_0(J) .
\]
Since $\mathcal R_0(K)$ is a compact metric space and $J\mapsto a_0(J)$ is a continuous map,
$A$ is actually a minimum. In particular, $A>0$ (one can of course consider Jacobi matrices with
zero $a$'s, but it's easy to see that such an operator
cannot be reflectionless on a positive measure set).

We must now show that if $J\in\mathcal R(K)$, then
$a_0(J)\ge A$, and we will have strict inequality here unless $J\in\mathcal R_0(K)$.
Since both $\mathcal R(K)$ and $\mathcal R_0(K)$ are shift invariant, this will yield the Theorem.

The $\xi$ function of a $J\in\mathcal R(K)$ satisfies
$\xi=1/2$ on $K$, and of course we have that $0\le\xi\le 1$. Fix a bounded component $I=(c,d)$ of the
open set $K^c$. We now claim that if we replace $\xi$ with
\[
\xi_1(t) = \begin{cases} \chi_{(d-g,d)}(t) & t\in I\\ \xi(t) & t\notin I \end{cases} ,
\]
where $g=\int_c^d\xi\, dt$, then $|H_1(x)|\le |H(x)|$ for almost every $x\notin (c,d)$.
More precisely, recall that we define $H(x)=\lim_{y\to 0+}H(x+iy)$ for $x\in\R$, and this
limit exists almost everywhere; we obtain the inequality for all $x\notin [c,d]$ for
which $H(x)$, $H_1(x)$ can be defined in this way.
The proof of this claim is easy; use the representation \eqref{2.3} to compare $H$ and
$H_1$. The statement is
also formulated as \cite[Lemma 2.5]{PR1}, and a detailed formal proof can be found there.

We now want to modify $\xi$ in this way on all bounded components of $K^c$, and we
also want to put $\xi=1$ to the left of $K$ and $\xi=0$ to the right of $K$. The plan is
to pass from $\xi$ to a new Krein function $\xi_0$ through a series of intermediate
functions $\xi_n$, and we expect that $|H_0(x)|\le |H(x)|$ for almost every $x\in K$,
because $|H_n(x)|$ goes down at each individual step.

To make this rigorous, it is best to work with the Hilbert transform
\[
(T\xi)(x) = \lim_{y\to 0+} \int_{|t-x|>y; |t|<R} \frac{\xi(t)\, dt}{t-x} ;
\]
here, $R>0$ is again chosen so large that $\sigma(J)\subset [-R,R]$.
We use the following basic facts; please consult \cite{CMT} for background information.
The limit defining the Hilbert transform
exists for almost every $x\in\R$, and
\[
|H(x)|=(x+R) e^{(T\xi)(x)}
\]
almost everywhere. Moreover, if $\xi_n\to\xi$ in $L^2(-R,R)$, then also $T\xi_n\to T\xi$ in $L^2(\R)$.

Now put $\widetilde{\xi}=1$ to the left of $K$, $\widetilde{\xi}=0$
to the right of $K$, and $\widetilde{\xi}=\xi$ otherwise.
It is obvious that $T\widetilde{\xi}\le T\xi$ (almost everywhere) on $K$. Let $\{ I_n \}_{n\ge 1}$
be a complete list of the bounded components of $K^c$; if there are only finitely many components,
then no limiting process is necessary and this part of the discussion becomes much easier.
Modify $\widetilde{\xi}$ on $I_1$ in the
way described above to obtain $\xi_1$. As we saw earlier, we then have that
$T\xi_1\le T\widetilde{\xi}$ almost everywhere on $K$. Continue in this way; construct $\xi_{n+1}$ from
$\xi_n$ by modifying this function on $I_{n+1}$. Clearly, $\xi_n$ converges in $L^1(-R,R)$
(and thus also in $L^2$, since $0\le\xi_n\le 1$) to a limiting function $\xi_0$.
The Hilbert transforms converge in $L^2$ and also
pointwise almost everywhere on $(-R,R)$: $(T\xi_n)(x)\to
(T\xi_0)(x)$ (it's not necessary to pass to a subsequence because this sequence
is decreasing on $K$ and eventually constant on $K^c$).
Hence $(T\xi_0)(x)\le (T\xi)(x)$ for almost every $x\in K$.
The Krein function $\xi_0$ is a step function on each gap $I_n$, and it jumps
from $0$ to $1$ (or not at all). Moreover, $\xi_0=1/2$ on $K$, and $\xi_0=1$
to the left of $K$ and $\xi_0=0$ to the right of $K$. In \cite{PR2},
we introduced the notation $X(K)$ for
the set of all Krein functions of this type, so we can summarize by writing
$\xi_0\in X(K)$. The significance of $X(K)$ comes from the fact that
$X(K)$ is exactly the collection of Krein functions
of operators from $\mathcal R_0(K)$; see \cite[Section 2]{PR2}
for this statement, especially the discussion of equation (2.6) of \cite{PR2}.

Recall that for any $J\in\mathcal R(K)$, we have that
\[
\pi\rho_{ac}(K) = \int_K |H(x)|\, dx ,
\]
where $\rho$ refers to the measure associated with the $H$ function of $J$, as in \eqref{2.4}.
Thus, in terms of $\rho$, what we have just shown says that if $J\in\mathcal R(K)$, then
there exists $\xi_0\in X(K)$ so that
\begin{equation}
\label{3.3}
\rho_{ac}(K)\ge \rho_{ac}^{(0)}(K) ,
\end{equation}
where $\rho^{(0)}$ is the measure associated with (the $H$ function of) $\xi_0$.
Now recall how the half line spectral measures $\nu_+$ were obtained from $\rho$: We have that
$\nu_+ = f\rho$, where $f=1/2$ Lebesgue almost everywhere on $K$ and $0\le f\le 1$, and
every such $f$ is admissible. Thus $\nu_+(\R)\ge (1/2)\rho_{ac}(K)$.
Moreover, if $\xi_0\in X(K)$ and we take $f=0$ off the support of $\rho_{ac}^{(0)}$,
then the corresponding
$\nu_+^{(0)}=(1/2)\chi_K\rho_{ac}^{(0)}$ belongs to a Jacobi matrix $J_0\in\mathcal R_0(K)$.
See again \cite[Section 2]{PR2} for this description of the spectral data of $\mathcal R_0(K)$.

The upshot of all this is that \eqref{3.3} implies that for arbitrary $J\in\mathcal R(K)$,
there exists $J_0\in\mathcal R_0(K)$ so that
\[
\nu_+(\R) \ge \nu_+^{(0)}(\R) .
\]
In other words, $a_0(J)\ge a_0(J_0)\ge A$, as we wanted to show.

It remains to discuss the case of equality. Clearly, this can only happen if $T\xi_0=T\xi$
almost everywhere on $K$ in the construction above, but at each individual step, when going from
$\xi_n$ to $\xi_{n+1}$, the Hilbert transform will decrease stricly almost everywhere on $K$ unless
$\xi_n=\xi_{n+1}$. Therefore, $T\xi_0=T\xi$ on $K$ implies that $\xi=\xi_0$,
and thus $\xi\in X(K)$. Moreover,
even for fixed $\rho$, the half line spectral measure
$\nu_+(\R)=a_0^2$ is of course minimized only by $\nu_+=(1/2)\chi_K\rho_{ac}$. As just explained,
this implies that $J\in \mathcal R_0(K)$.
\end{proof}
Finally, here's the promised quantitative
version of the result of Dombrowski and Simon-Spencer \cite{Dom,SimSp}.
These authors show by a decoupling argument that $\liminf a_n>0$ if $\sigma_{ac}(J_+)\not=\emptyset$.
\begin{Theorem}
Let $M\subset\R$ be a bounded Borel set of positive Lebesgue measure. Then there exists a constant
$A=A(M)>0$ such that
\begin{equation}
\label{3.4}
\liminf_{n\to\infty} a_n \ge A
\end{equation}
for every half line Jacobi matrix $J_+$ with $\Sigma_{ac}(J_+)\supset M$.
\end{Theorem}
\begin{proof}
Fix a compact subset $K\subset M$ of positive Lebesgue measure, and let $A=A(K)>0$ be the constant
from Theorem \ref{T3.1}.
Then $a_n(J)\ge A$ for all (whole line operators) $J\in\mathcal R(M)\subset\mathcal R(K)$.
Since $\omega(J_+)\subset\mathcal R(M)$ by \cite[Theorem 1.4]{Remac}, we obtain \eqref{3.4} from this.
\end{proof}

\end{document}